\documentclass[a4paper,reqno]{amsart}

\usepackage{amsmath,amssymb,amsthm}
\usepackage{mathrsfs}
\usepackage{indentfirst,color}
\usepackage[colorlinks,citecolor=red,linkcolor=blue,urlcolor=cyan]{hyperref}
\usepackage{graphicx}
\usepackage{tikz,tikz-cd}
\usepackage{bm}

\theoremstyle{plain}%
\newtheorem{theorem}{Theorem}[section]%
\newtheorem{proposition}[theorem]{Proposition}%
\newtheorem{lemma}[theorem]{Lemma}%
\newtheorem{example}[theorem]{Example}%
\newtheorem{corollary}[theorem]{Corollary}%

\theoremstyle{definition}%
\newtheorem{definition}[theorem]{Definition}%
\newtheorem{remark}[theorem]{Remark}%

\newcommand{\CC}{{\mathbb{C}}}%

\makeatletter
\newsavebox{\@brx}
\newcommand{\llangle}[1][]{\savebox{\@brx}{\(\m@th{#1\langle}\)}%
    \mathopen{\copy\@brx\kern-0.5\wd\@brx\usebox{\@brx}}}
\newcommand{\rrangle}[1][]{\savebox{\@brx}{\(\m@th{#1\rangle}\)}%
    \mathclose{\copy\@brx\kern-0.5\wd\@brx\usebox{\@brx}}}
\makeatother

\newcommand{\loc}{\textup{loc}}%
\newcommand{\sm}{\textup{sm}}%
\newcommand{\nd}{\textup{nd}}%
\newcommand{\image}{\textup{i}}%
\newcommand{\Hom}{\textup{Hom}}%
\renewcommand{\bar}[1]{\overline{#1}}%
\newcommand{\ddbar}{\image\partial\bar{\partial}}%

\renewcommand{\leq}{\leqslant}%
\renewcommand{\geq}{\geqslant}%

\numberwithin{equation}{section}

\begin{document}

\title[A logarithmic Bogomolov--Sommese vanishing theorem]{A logarithmic Bogomolov--Sommese vanishing theorem on compact K\"ahler manifolds}

\author[Zhi Li]{Zhi Li}
\address{Zhi Li: School of Mathematical Sciences and Key Laboratory of Mathematics and Information Networks (Ministry of Education), Beijing University of Posts and Telecommunications, Beijing 100876, China.}
\email{lizhi@amss.ac.cn, lizhi10@foxmail.com}

\author[Xiankui Meng]{Xiankui Meng}
\address{Xiankui Meng: School of Mathematical Sciences and Key Laboratory of Mathematics and Information Networks (Ministry of Education), Beijing University of Posts and Telecommunications, Beijing 100876, China.}
\email{mengxiankui@amss.ac.cn}

\author[Kai Pang]{Kai Pang}
\address{Kai Pang: Institute of Mathematics, Academy of Mathematics and Systems Science,
Chinese Academy of Sciences, Beijing 100190, China.}
\email{pangkai@amss.ac.cn}

\author[Chenghao Qing]{Chenghao Qing}
\address{Chenghao Qing: Yau Mathematical Sciences Center, Tsinghua University, Beijing 100084, China.}
\email{qingchenghao@amss.ac.cn}

\author[Xiangyu Zhou]{Xiangyu Zhou}
\address{Xiangyu Zhou: Institute of Mathematics, Academy of Mathematics and Systems Science, Chinese Academy of Sciences, Beijing 100190, China.}
\email{xyzhou@math.ac.cn}

\date{}

\thanks{The first and second author were partially supported by the National Natural Science Foundation of China (Grant No. 12271057) and by the National Key R\&D Program of China (Grant No. 2021YFA1002600). 
The fourth author was supported by the China Postdoctoral Science Foundation (Grant No. 2025M773087). 
The fifth author was supported by National Key R\&D Program of China (Grant No. 2021YFA1003100) and by the National Natural Science Foundation of China (Grant No. 12288201).
}

\begin{abstract}

    In this paper, we establish a logarithmic Bogomolov--Sommese vanishing theorem in terms of numerical dimension for pseudo-effective line bundles on compact K\"ahler manifolds. As an application, we obtain a rigidity result with vanishing second Chern class for logarithmic cotangent bundles by combining the vanishing theorem with a structure theorem of Iwai and Matsumura.
\end{abstract}



\maketitle

\section{Introduction}
The logarithmic Bogomolov--Sommese theorem bounds the Kodaira--Iitaka dimension of a rank-one subsheaf of $\Omega_X^p(\log D)$ by $p$. 
The classical version is the following theorem.

\begin{theorem}[\cite{Bog78,Bog80,SS85}]\label{Theorem:BS vanishing}
    Let $X$ be a projective manifold and $L$ a line bundle over $X$. 
    Let $D$ be a reduced simple normal crossing divisor on $X$. Then
    \begin{equation*}
        H^0(X,\Omega^p_X(\log D)\otimes L^{-1})=0 \quad \text{for every } p<\kappa(L),
    \end{equation*}
    where $\kappa(L)$ denotes the Kodaira--Iitaka dimension of $L$.
\end{theorem}

Note that Theorem~\ref{Theorem:BS vanishing} is meaningful only when $\kappa(L)>0$, in which case $L$ is $\mathbb{Q}$-effective, in particular, pseudo-effective.
Thus one only needs to consider pseudo-effective line bundles.
By applying the Calabi-Yau theorem, Mourougane (in the nef case) and Boucksom obtained a Bogomolov-type vanishing theorem on compact K\"ahler manifolds, which is useful in birational geometry.
\begin{theorem}[{\cite{Bou02}}]\label{Theorem:Boucksom}
    Let $X$ be a compact K\"ahler manifold and $L$ a pseudo-effective line bundle over $X$. 
    Then
    \begin{equation*}
        H^0(X,\Omega^p_X\otimes L^{-1})=0 \quad \text{for every } p<\nd(L),
    \end{equation*}
    where $\nd(L)$ denotes the numerical dimension of $L$.
\end{theorem}

The classical vanishing theorems of Bogomolov–Sommese (Theorem~\ref{Theorem:BS vanishing}) and Boucksom (Theorem~\ref{Theorem:Boucksom}) have been fundamental in birational geometry. The former establishes a precise link between the Kodaira–Iitaka dimension and logarithmic differential forms on projective manifolds, while the latter extends the vanishing to pseudo-effective line bundles on compact Kähler manifolds using the numerical dimension. Therefore, finding a logarithmic analogue of Boucksom's numerical criterion is a natural problem. Stimulated by these deep results and the problem, we establish the following theorem, which solves the problem and provides a unified numerical vanishing criterion for logarithmic forms in the Kähler setting.

\begin{theorem}\label{Main Theorem 1}
    Let $X$ be a compact K\"ahler manifold, $L$ a pseudo-effective line bundle over $X$ and $D$ a reduced simple normal crossing divisor on $X$. Then
    \begin{equation*}
        H^0(X,\Omega^p_X(\log D)\otimes L^{-1})=0 \quad \text{for every } p<\nd(L).
    \end{equation*}
\end{theorem}

Theorem~\ref{Main Theorem 1} recovers Theorem~\ref{Theorem:Boucksom} by taking $D=0$, and in the projective case, since $\kappa(L)\leqslant \nd(L)$, it means Theorem~\ref{Theorem:BS vanishing} be with a potentially larger vanishing range. By Serre duality, Theorem~\ref{Main Theorem 1} can be reformulated equivalently as follows.
\begin{corollary}\label{Corollary:Corollary without multiplier}
    Let $X$ be a compact K\"ahler manifold of dimension $n$ and $L$ a pseudo-effective line bundle over $X$.
    Let $D$ be a reduced simple normal crossing divisor on $X$. Then
    \begin{equation*}
        H^n(X,\Omega^p_X(\log D)\otimes L\otimes\mathcal{O}_X(-D))=0 \quad \text{for every } p\geq n-\nd(L)+1.
    \end{equation*}
    In particular, when $D=0$,
    \begin{equation*}
        H^n(X,\Omega^p_X\otimes L)=0 \quad \text{for every } p\geq n-\nd(L)+1.
    \end{equation*}
\end{corollary}
Via Siu's decomposition of currents, we can obtain a vanishing theorem with multiplier ideal sheaves.
\begin{corollary}\label{Corollary:Corollary with multiplier}
    Let $X$ be a compact K\"ahler manifold of dimension $n$ and $L$ a pseudo-effective line bundle over $X$.
    Let $D$ be a reduced simple normal crossing divisor on $X$. 
    Assume that $\{\alpha\}$ is a pseudo-effective class on $X$ and $c_1(L)-\{\alpha\}=\{\theta+\frac{\ddbar\psi}{2\pi}\}$, where $\theta$ is a smooth real closed $(1,1)$-form and $\psi$
    is a function on $X$ such that $\theta+\frac{\ddbar\psi}{2\pi}\geq0$ in the sense of currents.
    Then
    $$H^n(X,\Omega^p_X(\log D)\otimes L\otimes\mathcal{O}_X(-D)\otimes\mathcal{I}(\psi))=0 \quad\text{for } p\geq n-\nd(\{\alpha\})+1.$$
\end{corollary}
Here the numerical dimension of a pseudo-effective class is different from the numerical dimension of a closed positive $(1,1)$-current given by Cao in \cite{Cao14}. 
In general, $\nd(\{T\})\geq\nd(T)$ for a closed positive $(1,1)$-current $T$.
We refer to \cite{Bou02,EV92,Gra15,LMNWZ25,Mou98,SS85,Wat23,Wat26,Wu20} for some generalizations of Bogomolov-type vanishing theorems.

Finally, as an application of Theorem~\ref{Main Theorem 1} and the results in \cite{CP16,IM22}, we obtain the following result.

\begin{theorem}
    Let $X$ be a compact K\"ahler manifold and $D$ a reduced simple normal crossing divisor on $X$.
    Assume that $K_X+D$ is nef.
    Then the following are equivalent:
    \begin{enumerate}
        \item $c_2(\Omega^1_X(\log D))=0 \text{ in } H^{2,2}(X,\mathbb{R})$.
        \item $\Omega^1_X(\log D)$ is nef and $\nd(K_X+D)\leq1$.
    \end{enumerate}

\end{theorem}

The remaining part of this paper is organized as follows. 
In Section~\ref{Section:Preliminary}, we briefly recall some preliminary results. 
In Section~\ref{Section:Bogomolov--Sommese Vanishing}, we prove Theorem~\ref{Main Theorem 1} and Corollary~\ref{Corollary:Corollary with multiplier}.
We also provide a simple proof on projective manifolds via hyperplane induction.
In Section~\ref{Section:Applications}, we give several geometric consequences of the main theorem.

\section{Preliminaries} \label{Section:Preliminary}

In this section, we recall some basic results used in the proofs of the main results in the present paper.

\subsection{Sheaf of logarithmic differential forms}
We first recall here some basic properties of the sheaf of logarithmic differential forms (see \cite[Chapter 2]{EV92}).
A reduced divisor $D=\sum_{j=1}^{s}D_j$ on a compact K\"ahler manifold $X$ is called a simple normal crossing divisor if every irreducible component $D_j$ is smooth and all intersections are transverse.
The sheaf $\Omega_X^p(\log D)$ is the sheaf of germs of differential $p$-forms on $X$ with at most logarithmic poles along $D$.
Sections of the logarithmic sheaf $\Omega_X^p(\log D)$ on an open subset $U$ are 
$$\Gamma(U,\Omega_X^p(\log D)):=\{\alpha\in\Gamma(U,\Omega_X^p\otimes \mathcal{O}_X(D)) ~|~d\alpha\in\Gamma(U,\Omega_X^{p+1}\otimes \mathcal{O}_X(D)) \}.$$
It is easy to see that when $p=\dim X$, $\Omega_X^p(\log D)=K_X\otimes\mathcal{O}_X(D)$.
The sheaf $\Omega_X^p(\log D)$ is locally free. 
Let $L$ be a line bundle on $X$. Logarithmic Serre duality gives an isomorphism
$$H^q(X,\Omega_X^p(\log D)\otimes L)\cong H^{n-q}(X,\Omega_X^{n-p}(\log D)\otimes L^{-1}\otimes \mathcal{O}_X(D)^{-1})^*.$$

Set $Y=X\setminus D$.
We can choose a local coordinate chart $(W;z_1,\ldots,z_n)$ such that $D\cap W=\{z_1\cdots z_t=0 \}$ and $Y\cap W=W_r^\ast=(\Delta^\ast_r)^t\times(\Delta_r)^{n-t}$, 
where $\Delta_r$ is the open disk of radius $r$ centered at $0$ in $\mathbb{C}$ and $\Delta^\ast_r=\Delta_r\setminus\{0\}$. 
We consider a special K\"ahler metric on $Y$ with some asymptotic properties along $D$.
\begin{definition}
    We say that the metric $\omega_Y$ on $Y$ is of Poincar\'e type along $D$, if for each local coordinate chart $(W;z_1,\ldots,z_n)$ along $D$, $\omega_Y|_{W_r^\ast}$ is equivalent to the usual Poincar\'{e} type metric defined by 
    $$\omega_P=i\sum_{j=1}^{t}\frac{dz_j\wedge d\bar{z}_j}{|z_j|^2(\log|z_j|^2)^2}+i\sum_{j=t+1}^{n}dz_j\wedge d\bar{z}_j.$$
\end{definition}
According to \cite{Zuc79}, there is a K\"ahler metric $\omega_P$ on $Y$ which is of Poincar\'e type along $D$.

\subsection{Singular Hermitian metrics on line bundles}

Let $(X,\omega)$ be a Hermitian manifold of dimension $n$.
A singular Hermitian metric $h$ on a holomorphic line bundle $L\to X$ is simply a Hermitian metric which can be expressed locally as
$e^{-\varphi_U}$ on $U$ such that $\varphi_U$ is $L^1_\loc$, where
$U\subset X$ is a local coordinate chart such that $L|_U\simeq U\times\mathbb{C}.$
It has a well-defined curvature current $\image\Theta_{L,h}:=\ddbar\varphi_U.$

\begin{definition}
    A function $\varphi:X\rightarrow [-\infty,+\infty)$ is
    said to be quasi-plurisubharmonic (quasi-psh for short) if $\varphi$ is locally the sum of a plurisubharmonic function and a smooth function 
    (or equivalently, if $\ddbar\varphi$ is locally bounded from below). 

    If $\varphi$ is a quasi-psh function on $X$, the multiplier ideal sheaf $\mathcal{I}(\varphi)$ is the ideal subsheaf of $\mathcal{O}_X$ defined by
    $$\mathcal{I}(\varphi)_x=\{f\in\mathcal{O}_{X,x}~|~\exists\ U\ni x\  \text{such that}\ \int_U|f|^2e^{-\varphi}d\lambda<+\infty  \},$$
    where $U$ is an open coordinate neighborhood of $x$ and $d\lambda$ is the standard Lebesgue measure in $\mathbb{C}^n.$
    
    It is easy to see that associated to a singular Hermitian metric $h$ on $L$ satisfying
    $\image\Theta_{L,h}\geq\gamma$ 
    for some smooth real $(1,1)$-form $\gamma$ in the sense of currents, 
    there is a well-defined multiplier ideal sheaf $\mathcal{I}(h)$ on $X$ and $\mathcal{I}(h)$ is coherent (\cite{Nad90}).
\end{definition}

We now recall some definitions of positivity of $(1,1)$-classes and line bundles (see \cite{DemSmall}).
\begin{definition}
    Let $(X,\omega)$ be a compact K\"ahler manifold of dimension $n$ and $\{\alpha\}\in H^{1,1}(X,\mathbb{R})$ a $(1,1)$-class on $X$. Let $L$ be a line bundle on $X$.
    \begin{itemize}
        \item[(1)] $\{\alpha\}$ is said to be nef if for every $\varepsilon>0$, 
        there exists a smooth function $\varphi_\varepsilon$ on $X$ such that $\alpha+\ddbar \varphi_\varepsilon\geq-\varepsilon\omega$.
        One defines the numerical dimension of $\{\alpha\}$ to be
        $$\nd(\{\alpha\})=\max\{k=0,\ldots,n~|~\{\alpha\}^k\neq0 \text{ in } H^{2k}(X,\mathbb{R})\}.$$
        $L$ is said to be nef if $c_1(L)$ is a nef class.
        In this case, we define its numerical
        dimension by $\nd(L)=\nd(c_1(L))$.
        A vector bundle $E$ is called nef if the line bundle $\mathcal{O}_E(1)$ on $\mathbb{P}(E)$, the projective bundle of hyperplanes in the fibres of $E$, is nef.
        \item[(2)] $\{\alpha\}$ is said to be pseudo-effective if there is a quasi-psh function $\varphi$ such that $\alpha+\ddbar\varphi\geq 0$ in the sense of currents.
        One defines the numerical dimension of $\{\alpha\}$ to be
        $$\nd(\{\alpha\})=\max\{k=0,\ldots,n~|~\langle \{\alpha\}^k \rangle\neq0 \}.$$
        Here $\langle \{\alpha\}^k \rangle:=\lim\limits_{\delta\to0}\{\langle T^k_{\min,\delta\omega} \rangle\}$ 
        where $T^k_{\min,\delta\omega}$ is the positive current with minimal singularity in the class $\{\alpha+\delta\omega\}$ and $\langle T^k_{\min,\delta\omega} \rangle$ is the non-pluripolar product.
        We refer to \cite{BEGZ10} for more precise argument.
        $L$ is said to be pseudo-effective if $c_1(L)$ is a pseudo-effective class. 
        In this case, we define its numerical
        dimension by $\nd(L)=\nd(c_1(L))$, and there is a singular Hermitian metric $h$ on $L$ such that $\image\Theta_{L,h}\geq 0$ in the sense of currents.
        \item[(3)] $\{\alpha\}$ is said to be big if there is a quasi-psh function $\varphi$ such that $\alpha+\ddbar\varphi\geq \varepsilon\omega$ for some positive continuous function $\varepsilon$ in the sense of currents.
        $L$ is said to be big if $c_1(L)$ is a big class. 
        In this case, there is a singular Hermitian metric $h$ on $L$ such that $\image\Theta_{L,h}\geq \varepsilon\omega$ in the sense of currents. 
        Note that a nef line bundle is big if its numerical dimension is $n$.
    \end{itemize}
\end{definition}

If $L$ is pseudo-effective line bundle over $X$, there is an
equivalent definition of the numerical dimension of the line bundle $L$ (see \cite[Proposition 2]{Wu20}):
    \begin{align*}
        \nd(L):=\max_k\Bigl\{
        &\text{ there exists a constant }c>0, \text{ such that for any } \varepsilon>0,\\
        &\text{ there is a singular metric }h_\varepsilon  \text{ with analytic singularities satisfying }\\
        &\image\Theta_{L,h_\varepsilon}\geq-\varepsilon\omega \text{ and } 
        \int_{X\setminus Z_\varepsilon}(\image\Theta_{L,h_\varepsilon}+\varepsilon\omega)^k\wedge\omega^{n-k}\geq c
        \Bigr\},
    \end{align*}
    where $Z_\varepsilon$ is the singular set of $h_\varepsilon$.
Then we can obtain well-behaved metrics on log resolutions.
\begin{lemma}\label{Lemma:Construct suitable metrics}
    Let $(X,\omega)$ be a compact K\"ahler manifold of dimension $n$ and $L$ a pseudo-effective line bundle over $X$ with $\nd(L)=k>p$.
    Let $D$ be a reduced simple normal crossing divisor on $X$.
    There exist constants $c_0,C_0>0$ such that for every $\varepsilon$ sufficiently small, we have the following properties.
    \begin{enumerate}
        \item There exists a log resolution 
        $$\pi_\varepsilon:\widetilde{X}_\varepsilon\longrightarrow X$$
        and a reduced simple normal crossing divisor $B_\varepsilon$ such that $B_\varepsilon$ contains the total transform of $D$ and the exceptional divisor.
        Moreover, $\pi_\varepsilon$ induces a natural morphism from $\Omega^p_X(\log D)$ to $\Omega^p_{\widetilde{X}_\varepsilon}(\log B_\varepsilon)$.
        \item There exists a singular Hermitian metric $\widetilde{h}_\varepsilon$ on line bundle $\pi^*_\varepsilon L$ 
        with divisorial singularities supported on $B_\varepsilon$ and smooth on $$Y_\varepsilon:=\widetilde{X}_\varepsilon\setminus B_\varepsilon.$$
        \item There exist K\"ahler forms $\gamma_\varepsilon$ and $\Omega_\varepsilon$ on $\widetilde{X}_\varepsilon$ such that, on $Y_\varepsilon$,
        \begin{equation}
            \image\Theta_{\pi^*_\varepsilon L,\widetilde{h}_\varepsilon}=\gamma_\varepsilon-3\varepsilon\pi^*_\varepsilon\omega 
            \quad \text{and} \quad \Omega_\varepsilon\geq \pi^*_\varepsilon\omega.
        \end{equation}
        \item The following uniform estimates hold:
        \begin{equation}
            V_\varepsilon:=\int_{\widetilde{X}_\varepsilon}\gamma_\varepsilon^n\geq c_0\varepsilon^{n-k} 
            \quad \text{and} \quad
            \int_{\widetilde{X}_\varepsilon}\gamma_\varepsilon^{p}\wedge\Omega_\varepsilon^{n-p}\leq C_0.
        \end{equation}
    \end{enumerate}
\end{lemma}
\begin{proof}
    By the equivalent definition of $\nd(L)$, there is a constant $c_{1}>0$ such that 
    for every $\varepsilon$ sufficiently small, there is a singular metric $h_\varepsilon$ with analytic singularities satisfying
    \begin{equation*}
        T_\varepsilon:=\image\Theta_{L,h_\varepsilon}+2\varepsilon\omega\geq 0 
        \quad \text{and} \quad
        \int_{X\setminus Z_\varepsilon} T_\varepsilon^k\wedge\omega^{n-k}\geq c_1>0,
    \end{equation*}
    where $Z_\varepsilon$ is the singular set of $h_\varepsilon$.
    On the other hand, consider the current $A_\varepsilon:=T_\varepsilon+\varepsilon\omega$.
    A current with minimal singularities in class $2\pi c_1(L)+3\varepsilon[\omega]$ is less singular than $A_\varepsilon$.
    Then by the definition of numerical dimension of $c_1(L)$ and the monotonicity of non-pluripolar
    products (refer to Theorem 1.16 and the explanations that follows Definition 1.17 in \cite{BEGZ10}), there is a constant $C_1$ such that
    \begin{equation}
        \int_{X\setminus Z_\varepsilon} A_\varepsilon^p\wedge\omega^{n-p}\leq C_1.
    \end{equation}
    We also have 
    \begin{equation}
        \int_{X\setminus Z_\varepsilon} A_\varepsilon^n\geq 
        \binom{n}{k}\varepsilon^{n-k}\int_{X\setminus Z_\varepsilon} T_\varepsilon^k\wedge\omega^{n-k}\geq c_2\varepsilon^{n-k}.
    \end{equation}
    
    Taking a suitable log resolution of $(X,D)$ and the singularities of $h_\varepsilon$, we obtain
    \begin{equation}
        \pi_\varepsilon:\widetilde{X}_\varepsilon\longrightarrow X 
        \quad \text{with} \quad \pi^*_\varepsilon A_\varepsilon=[E_\varepsilon]+\beta_\varepsilon
        \quad \text{and } \beta_\varepsilon\geq\varepsilon\pi^*_\varepsilon\omega,
    \end{equation}
    where $E_\varepsilon$ is an effective divisor and $\beta_\varepsilon$ is a smooth closed semi-positive $(1,1)$-form.
    Then for any $0\leq q\leq n$, we have 
    \begin{equation}
        \int_{\widetilde{X}_\varepsilon}\beta_\varepsilon^q\wedge(\pi^*_\varepsilon\omega)^{n-q}=
        \int_{X\setminus Z_\varepsilon}A_\varepsilon^q\wedge\omega^{n-q}.
    \end{equation}
    Denote by $F_{\varepsilon,1},\ldots,F_{\varepsilon,N_\varepsilon}$ the exceptional prime divisors.
    For sufficiently small $a_{\varepsilon,j}>0$, the class $[\beta_\varepsilon]-\sum_{j}a_{\varepsilon,j}[F_{\varepsilon,j}]$ is a K\"ahler class.
    Hence we can choose $a_{\varepsilon,j}$ so small that there is a K\"ahler representative $\gamma_\varepsilon$ satisfying
    \begin{equation}
        \int_{\widetilde{X}_\varepsilon}\gamma_\varepsilon^n\geq\frac{1}{2}\int\beta_\varepsilon^n\geq\frac{c_2}{2}\varepsilon^{n-k}
        \quad\text{and}\quad
        \int_{\widetilde{X}_\varepsilon}\gamma_\varepsilon^p\wedge\pi^*_\varepsilon\omega^{n-p}\leq C_1+1.
    \end{equation}
    Let $\widetilde{\omega}_\varepsilon$ be a fixed K\"ahler metric on $\widetilde{X}_\varepsilon$.
    Then for any $\varepsilon>0$, there exists a constant $t=t(\varepsilon)$ small enough such that $\Omega_\varepsilon:=\pi^*_\varepsilon\omega+t\widetilde{\omega}_\varepsilon$ is a K\"ahler metric and
    \begin{equation}
        \int_{\widetilde{X}_\varepsilon}\gamma_\varepsilon^p\wedge\Omega_\varepsilon^{n-p}\leq C_1+2.
    \end{equation}
    
    We finally construct the metric $\widetilde{h}_\varepsilon$.
    Let $g_{\varepsilon,j}$ be a fixed smooth Hermitian metric on $\mathcal{O}_{\widetilde{X}_\varepsilon}(F_{\varepsilon,j})$ and $s_{\varepsilon,j}$ the canonical section of $\mathcal{O}_{\widetilde{X}_\varepsilon}(F_{\varepsilon,j})$.
    Since $\gamma_\varepsilon$ is a K\"ahler representative of $[\beta_\varepsilon]-\sum_{j}a_{\varepsilon,j}[F_{\varepsilon,j}]$,
    there is a smooth function $\phi_\varepsilon$ such that
    \begin{equation}
        \gamma_\varepsilon=\beta_\varepsilon-\frac{\sqrt{-1}}{2\pi}\sum_{j}a_{\varepsilon,j}\Theta_{g_{\varepsilon,j}}+\ddbar\phi_\varepsilon.
    \end{equation}
    Define a global quasi-plurisubharmonic function
    $$\chi_\varepsilon:=\phi_\varepsilon+\frac{1}{2\pi}\sum_{j}a_{\varepsilon,j}\log|s_{\varepsilon,j}|^2_{g_{\varepsilon,j}}$$
    and a singular Hermitian metric on $\pi^*_\varepsilon L$ by
    $$\widetilde{h}_\varepsilon:=\pi^*_\varepsilon h_\varepsilon\cdot e^{-\chi_\varepsilon}.$$
    Then by Poincar\'e--Lelong formula, 
    \begin{align*}
        \image\Theta_{\pi^*_\varepsilon L,\widetilde{h}_\varepsilon}&=\pi^*_\varepsilon\image\Theta_{L,h_\varepsilon}+\ddbar\chi_\varepsilon \\
        &=\pi^*_\varepsilon(A_\varepsilon-3\varepsilon\omega)+\ddbar\phi_\varepsilon+\sum_{j}a_{\varepsilon,j}[F_{\varepsilon,j}]-\frac{\sqrt{-1}}{2\pi}\sum_{j}a_{\varepsilon,j}\Theta_{g_{\varepsilon,j}} \\
        &=[E_\varepsilon]+\sum_{j}a_{\varepsilon,j}[F_{\varepsilon,j}]+\gamma_\varepsilon-3\varepsilon\pi^*_\varepsilon\omega.
    \end{align*}
    In particular, 
    \begin{equation}
        \image\Theta_{\pi^*_\varepsilon L,\widetilde{h}_\varepsilon}=\gamma_\varepsilon-3\varepsilon\pi^*_\varepsilon\omega
    \end{equation}
    on $Y_\varepsilon$.
    
    This completes the proof of Lemma~\ref{Lemma:Construct suitable metrics}.
\end{proof}

Let us recall an important property of multiplier ideal sheaves used by Demailly and Peternell to finish this subsection.
\begin{proposition}[{\cite[Proposition~3.2]{DP03}}]\label{Proposition:Demailly--Peternell}
    Let $(L,h)$ be a pseudo-effective line bundle over a compact K\"ahler manifold $X$.
    Let $$\frac{\sqrt{-1}}{2\pi}\Theta_{L,h}=\sum_{j=1}^{\infty}\lambda_jD_j+R$$
    be the Siu decomposition of the closed positive $(1,1)$-current $\frac{\image}{2\pi}\Theta_{L,h}$ as a countable sum of effective divisors and of a $(1,1)$-current $R$ such that the Lelong sublevel sets $E_c(R)$, $c>0$, all have codimension two.
    Then we have the inclusion of sheaves
    $$\mathcal{I}(h)\subset\mathcal{O}_X(-\sum_{j}\lfloor \lambda_j \rfloor D_j),$$
    and equality holds on $X\setminus Z$ where $Z$ is an analytic subset of $X$ whose components all have codimension at least two.
\end{proposition}

\subsection{Degenerate Monge--Amp\`ere equation}

Solving Monge--Amp\`ere equations is nowadays a standard technique when considering questions related to numerical dimensions.
We use Demailly and Pali's result (reformulated and simplified here) on the existence and higher order regularity of solutions to degenerate Monge--Amp\`ere equations.
\begin{theorem}[{\cite[Theorem~6.1,\,6.2]{DP10}}]\label{Theorem:Demailly--Pali}
    Let $X$ be a compact connected K\"ahler manifold of complex dimension $n\geq2$, $\omega$ be a smooth K\"ahler form and $\Omega^n$ be a smooth volume form.
    Consider also $\sigma_j\in H^0(X,E_j)$, $\tau_r\in H^0(X,F_r)$, $j=1,\ldots,N$, $r=1,\ldots,M$ 
    be non-identically zero holomorphic sections of some holomorphic vector bundles over $X$,  such that the integral condition
    \begin{equation}
        \int_X\prod_{j=1}^N|\sigma_j|^{2l_j}\prod_{r=1}^{M}|\tau_r|^{-2h_r}\Omega^n=\int_X\omega^n
    \end{equation}
    holds for certain real numbers $l_j\geq0,h_r\geq0$.
    Then there exists a unique solution $\varphi$ of the degenerate complex Monge–Amp\`ere equation
    \begin{equation}
        (\omega+\ddbar\varphi)^n=e^{\lambda\varphi}\prod_{j=1}^N|\sigma_j|^{2l_j}\prod_{r=1}^{M}|\tau_r|^{-2h_r}\Omega^n, \quad \lambda\geq0,
    \end{equation}
    which in the case $\lambda=0$ is normalized by $\sup_X\varphi=0$.
    Consider the complex analytic sets
    $$S_1:=\bigcup_r\{\tau_r=0\}, \quad S_2=S_1\cup \left(\bigcup_j \{\sigma_j=0\}\right). $$
    Then $\varphi\in C^0(X)\cap C^{1,1}(X\setminus S_1) \cap C^{\infty}(X\setminus S_2)$.
    In particular, $\varphi\in L^\infty(X)$.
    
    Moreover, denote $f=\prod_{j=1}^N|\sigma_j|^{2l_j}\prod_{r=1}^{M}|\tau_r|^{-2h_r}$.
    For $\delta>0$ satisfying $f\in L\log^{n+\delta}L(X)$,
    the solution $\varphi$ satisfies
    \begin{equation}
        \|\varphi\|_{L^\infty(X)}\leq C(\delta,\omega,\Omega^n)I_{\omega,\delta}(f)^{\frac{n}{\delta}}+1,
    \end{equation}
    where 
    $$I_{\omega,\delta}(f):=\{\omega\}^{-n}\int_Xf\log^{n+\delta}(e+\{\omega\}^{-n}f)\Omega^n.$$
    
\end{theorem}
We refer to \cite[Section~2]{DP10} for some basic definitions and facts about Orlicz spaces.

\section{Proof of the main results} \label{Section:Bogomolov--Sommese Vanishing}

We proceed to prove Theorem~\ref{Main Theorem 1}. Our argument combines a refinement of the method in \cite{Bou02,Wu20} with the solution of the degenerate Monge–Ampère equation established in \cite[Theorem~6.1, 6.2]{DP10}.

\begin{proof}[Proof of Theorem~\ref{Main Theorem 1}]

    We prove the desired conclusion by contradiction.
    Assume that for some $p<k=\nd(L)$ there is a nonzero section
    $$u\in H^0(X,\Omega^p_X(\log D)\otimes L^{-1}).$$
    
    \medskip
    \noindent\textbf{Step 1: basic settings and constructions.}
    
    For each $\varepsilon$ sufficiently small, we use the construction in Lemma~\ref{Lemma:Construct suitable metrics}:
    $$\pi_\varepsilon:\widetilde{X}_\varepsilon\longrightarrow X, \quad B_\varepsilon=\sum_{j=1}^{N}B_{\varepsilon,j},
    \quad Y_\varepsilon=\widetilde{X}_\varepsilon\setminus B_\varepsilon \quad \text{and } \gamma_\varepsilon, \Omega_\varepsilon.$$
    The pullback of $u$ is a nonzero section
    $$v:=\pi^*_\varepsilon u\in H^0(\widetilde{X}_\varepsilon,\Omega_{\widetilde{X}_\varepsilon}^p(\log B_\varepsilon)\otimes\pi^*_\varepsilon L^{-1}).$$
    
    Choose smooth Hermitian metrics $g_j$ on $\mathcal{O}_{\widetilde{X}_\varepsilon}(B_{\varepsilon,j})$ such that
    $|s_j|^2_{g_j}<e^{-1}$
    where $s_j$ are canonical sections of $\mathcal{O}_{\widetilde{X}_\varepsilon}(B_{\varepsilon,j})$.
    Choose a constant $C_\varepsilon>0$ such that
    $$-C_\varepsilon\Omega_\varepsilon \leq \sum_{j=1}^N\sqrt{-1}\Theta_{\mathcal{O}_{\widetilde{X}_\varepsilon}(B_{\varepsilon,j}),g_j}\leq C_\varepsilon\Omega_\varepsilon.$$
    Choose a sufficiently small constant $0<\tau_\varepsilon<1/2$ such that 
    $$\ddbar\log w_\varepsilon\geq-\varepsilon\Omega_\varepsilon, \quad \text{where } w_\varepsilon:=\prod_{j=1}^N|s_j|_{g_j}^{2\tau_\varepsilon}.$$
    Let $s_B:=s_1\otimes\cdots\otimes s_N$ and $g_B:=g_1\otimes\cdots\otimes g_N$.
    The natural inclusion $\Omega_{\widetilde{X}_\varepsilon}^p(\log B_\varepsilon)\hookrightarrow\Omega_{\widetilde{X}_\varepsilon}^p\otimes\mathcal{O}_{\widetilde{X}_\varepsilon}(B_\varepsilon)$,
    induces a holomorphic section
    $$\widetilde{v}\in H^0(\widetilde{X}_\varepsilon,\Omega_{\widetilde{X}_\varepsilon}^p\otimes\mathcal{O}_{\widetilde{X}_\varepsilon}(B_\varepsilon)\otimes\pi^*_\varepsilon L^{-1}).$$
    Note that $\mathcal{O}_{\widetilde{X}_\varepsilon}(B_\varepsilon)$ is trivial on $Y_\varepsilon$.
    Therefore, $\Omega_{\widetilde{X}_\varepsilon}^p\otimes\mathcal{O}_{\widetilde{X}_\varepsilon}(B_\varepsilon)|_{Y_\varepsilon}
    \cong\Omega_{\widetilde{X}_\varepsilon}^p|_{Y_\varepsilon}$ via $\widetilde{v}=v\otimes s_B$, and
    one has the norm identity
    $$|\widetilde{v}_{\Omega_\varepsilon,g_B,h_{\sm}^{-1}}|=|v|_{\Omega_\varepsilon,h_{\sm}^{-1}}|s_B|_{g_B}$$
    for every smooth metric $h_\sm$ on $\pi^*_\varepsilon L$.
    
    For $0<\eta\leqslant 1$, consider the holomorphic section
    $$\sigma_\eta:=(\widetilde{v},\eta s_B)\in H^0\left(\widetilde{X}_\varepsilon,(\Omega_{\widetilde{X}_\varepsilon}^p\otimes\mathcal{O}_{\widetilde{X}_\varepsilon}(B_\varepsilon)\otimes\pi^*_\varepsilon L^{-1})\oplus\mathcal{O}_{\widetilde{X}_\varepsilon}(B_\varepsilon)\right).$$
    It is easy to see $\sigma_\eta$ is nowhere vanishing on $Y_\varepsilon$.
    Write the metric with divisorial singularities supported on $B_\varepsilon$ constructed in Lemma~\ref{Lemma:Construct suitable metrics} as 
    $$\widetilde{h}_\varepsilon=h_\sm e^{-\rho_\varepsilon}.$$
    There is a smooth function $f_\varepsilon$ on $\widetilde{X}_\varepsilon$ such that
    $$\rho_\varepsilon=f_\varepsilon+\sum_{j=1}^Nc_j \log |s_j|^2_{g_j}, \quad c_j\geq 0.$$       
    Define on $Y_\varepsilon$
    \begin{equation}
        \begin{aligned}
            F_\eta&:=e^{\rho_\varepsilon}w_\varepsilon(|v|_{\Omega_\varepsilon,h_\sm^{-1}}^2+\eta^2) \\
            &=e^{f_\varepsilon}|\sigma_\eta|^2_{h_\oplus}\frac{\prod_j|s_j|_{g_j}^{2c_j}}{\prod_j|s_j|_{g_j}^{2(1-\tau_\varepsilon)}},
        \end{aligned}
    \end{equation}
    where $h_\oplus$ is the metric on the direct sum bundle induced by $\Omega_\varepsilon,g_B$, and $h_\sm^{-1}$.
    Then for any $\varepsilon>0$, there is $q_\varepsilon>1$ such that
    \begin{equation}
        \sup_{0<\eta\leq1}\|F_\eta\|_{L^{q_\varepsilon}(\widetilde{X}_\varepsilon,\Omega_\varepsilon^n)}<+\infty.
    \end{equation}
    Indeed, it is enough to choose $q_\varepsilon\in(1,(1-\tau_\varepsilon)^{-1})$.
    
    \medskip
    \noindent\textbf{Step 2: solutions of certain Monge--Amp\`ere equations.}
    
    H\"older inequality implies that $\int_{\widetilde{X}_\varepsilon}F_\eta\Omega_\varepsilon^n$ are uniformly bounded above for $0<\eta\leq1$.
    By $F_\eta>F_0\not\equiv0$, they are also bounded below by a positive constant.
    Define 
    $$c_\eta:=\frac{\int_{\widetilde{X}_\varepsilon}\gamma_\varepsilon^n}{\int_{\widetilde{X}_\varepsilon}F_\eta\Omega_\varepsilon^n} 
    \quad \text{and} \quad \widetilde{F}_\eta:=c_\eta F_\eta.$$
    Then $$\sup_{0<\eta\leq1}\|\widetilde{F}_\eta\|_{L^{q_\varepsilon}(\widetilde{X}_\varepsilon,\Omega_\varepsilon^n)}<+\infty.$$
    In particular, for some $\delta_\varepsilon>0$, the functions $\widetilde{F}_\eta$ have uniformly bounded $L\log^{n+\delta_\varepsilon}L$-norms.
    
    By applying Demailly--Pali's result (Theorem~\ref{Theorem:Demailly--Pali}) to the normalized density $\widetilde{F}_\eta$ with $\lambda=1$, 
    there is a unique bounded $\gamma_\varepsilon$-psh solution $\psi_\eta$ of 
    \begin{equation*}
        \omega_{\varepsilon,\eta}^n=e^{\psi_\eta}\widetilde{F}_\eta\Omega_\varepsilon^n, \quad \text{where } \omega_{\varepsilon,\eta}=\gamma_\varepsilon+\ddbar\psi_\eta.
    \end{equation*}
    Moreover, $\psi_\eta\in L^\infty(\widetilde{X}_\varepsilon)\cap C^\infty(Y_\varepsilon)$, $\omega_{\varepsilon,\eta}$ is a K\"ahler form on $Y_\varepsilon$, and for any fixed $\varepsilon>0$,
    $$\sup_{0<\eta\leq1}\|\psi_\eta\|_{L^\infty(\widetilde{X}_\varepsilon)}<+\infty.$$
    Set $\varphi_\eta:=\psi_\eta+\log c_\eta$.
    Then we have 
    $$\omega_{\varepsilon,\eta}^n=e^{\varphi_\eta}F_\eta\Omega_\varepsilon^n 
    \quad \text{and} \quad
    \sup_{0<\eta\leq1}\|\varphi_\eta\|_{L^\infty(\widetilde{X}_\varepsilon)}<+\infty.$$
    Here $\varphi_{\eta}$ is locally bounded on $\widetilde{X}_\varepsilon$, thus the wedge product is well-defined in the sense of Bedford--Taylor.
    
    Define 
    $$G_{\varepsilon,\eta}:=e^{\rho_\varepsilon+\varphi_\eta}w_\varepsilon|v|^2_{\Omega_\varepsilon,h_\sm^{-1}}
    \quad \text{and} \quad
    R_{\varepsilon,\eta}:=\eta^2e^{\rho_\varepsilon+\varphi_\eta}w_\varepsilon.$$
    Then 
    $$\omega_{\varepsilon,\eta}^n=(G_{\varepsilon,\eta}+R_{\varepsilon,\eta})\Omega_\varepsilon^n
    \quad \text{and} \quad
    \int_{\widetilde{X}_\varepsilon}(G_{\varepsilon,\eta}+R_{\varepsilon,\eta})\Omega_\varepsilon^n=\int_{\widetilde{X}_\varepsilon}\gamma_\varepsilon^n.$$
    Note that $e^{\rho_\varepsilon}\in L^1(\widetilde{X}_\varepsilon)$ and the $L^\infty$-norms of $\varphi_\eta$ are uniformly bounded. We have 
    $$\int_{\widetilde{X}_\varepsilon}R_{\varepsilon,\eta}\Omega_\varepsilon^n\rightarrow 0 \quad \text{as } \eta\rightarrow0.$$
    Therefore, we can choose $\eta=\eta(\varepsilon)>0$ small enough such that
    $$M_\varepsilon:=\int_{\widetilde{X}_\varepsilon}G_{\varepsilon,\eta(\varepsilon)}\Omega_\varepsilon^n\geq\frac{1}{2}\int_{\widetilde{X}_\varepsilon}\gamma_\varepsilon^n.$$
    From now on, $\eta(\varepsilon)$ is fixed for a given $\varepsilon$. 
    Thus we write $\varphi_\varepsilon,\omega_\varepsilon,G_\varepsilon$ and $R_\varepsilon$.
    
    \medskip
    \noindent\textbf{Step 3: estimates via Bochner--Kodaira--Nakano identity.}
    
    Define a metric on $\pi^*_\varepsilon L$ by 
    $$\widehat{h}_\varepsilon:=\widetilde{h}_\varepsilon e^{-\varphi_\varepsilon}.$$
    Define a metric on $\pi^*_\varepsilon L^{-1}$ by
    $$H_\varepsilon:=\widehat{h}_\varepsilon^{-1}w_\varepsilon=h_\sm^{-1} e^{\rho_\varepsilon+\varphi_\varepsilon}w_\varepsilon.$$
    Then by the construction, 
    $|v|^2_{\Omega_\varepsilon,H_\varepsilon}=G_\varepsilon$.
    On $Y_\varepsilon$, we have 
    \begin{equation}
        \begin{aligned}
            \sqrt{-1}\Theta_{\pi^*_\varepsilon L,H_\varepsilon^{-1}}&=\omega_\varepsilon-3\varepsilon\pi_\varepsilon^*\omega+\ddbar\log w_\varepsilon \\
            &\geq \omega_\varepsilon-4\varepsilon\Omega_\varepsilon.
        \end{aligned}
    \end{equation}
    Let $\Omega_P$ be a complete Poincar\'e-type K\"ahler form on $Y_\varepsilon$ and consider $\Omega_{\varepsilon,\delta}:=\Omega_\varepsilon+\delta\Omega_P, 0<\delta\leq1$.
    We now show that $v$ is $L^2$-integrable on $Y_\varepsilon$ with respect to $\Omega_{\varepsilon,\delta}$ and $H_\varepsilon$ for every $0<\delta<1$. 
    
    Let $(W;z_1,\ldots,z_n)$ be a local coordinate chart of $\widetilde{X}_\varepsilon$ along $B_\varepsilon$ such that $W\cap B_\varepsilon=\{z_1\cdots z_t=0\}$. 
    Let $e$ be a trivialization section of $\pi^*_\varepsilon L^{-1}$ on $W$ such that $\frac{1}{2}<|e|_{h_\sm^{-1}}<1$ over $W$. 
    Denote $$\xi_j=\frac{1}{z_j}dz_j \quad \text{for } 1\leq j\leq t, \quad \xi_j=dz_j \quad \text{for } t+1\leq j\leq n.$$ 
    Write $v=\sum_{|I|=p}v_I\xi_{i_1}\wedge\cdots\wedge\xi_{i_p}\otimes e$ where $I=(i_1,\ldots,i_p)$ is a multi-index with $i_1 < \cdots< i_p$ and $v_I(z)$ is a holomorphic function on $W$.
    Recall that 
    $$\Omega_P \sim \sum_{j=1}^{t} \frac{\sqrt{-1}\, dz_j \wedge d\bar z_j}{|z_j|^2 (\log |z_j|^2)^2} + \sum_{j=t+1}^{n} \sqrt{-1}\, dz_j \wedge d\bar z_j,$$
    on $W$. Then 
    $\Omega_{\varepsilon,\delta} = \Omega_\varepsilon + \delta \Omega_P$ satisfies 
    $$\Omega_{\varepsilon,\delta} \sim \sum_{j=1}^{t} \frac{\sqrt{-1}\, dz_j \wedge d\bar z_j}{|z_j|^2 (\log |z_j|^2)^2} + \sum_{j=t+1}^{n} \sqrt{-1}\, dz_j \wedge d\bar z_j ,$$
    and hence 
    $$(\Omega_\varepsilon+\delta\Omega_P)^n \sim \prod_{j=1}^{t} \frac{1}{|z_j|^2 (\log |z_j|^2)^2} \, dV_{\mathbb C^n} .$$
    Note that 
    $$|\xi_I|^2_{\Omega_{\varepsilon,\delta}} \sim \prod_{j \in I, j\le t} (\log |z_j|^2)^2.$$
    Therefore,
    $$|\xi_I|^2_{\Omega_{\varepsilon,\delta}}  (\Omega_\varepsilon+\delta\Omega_P)^n\leqslant \widetilde{C} \prod_{j=1}^t |z_j|^{-2} \, dV_{\mathbb C^n} .$$
    Thus,
    
    \begin{align*}
         \int_{W}|v|^2_{\Omega_{\varepsilon,\delta},H_\varepsilon}dV_{\Omega_{\varepsilon,\delta}}&\leq C\int_{W}\sum_{I}|v_I|^2|\xi_{i_1}\wedge\cdots\xi_{i_p}|_{\Omega_{\varepsilon,\delta}}^2e^{\rho_\varepsilon+\varphi_\varepsilon}w_\varepsilon(\Omega_\varepsilon+\delta\Omega_P)^n   \\
        &\leq C'\int_W\prod_{j=1}^t|z_j|^{2c_j+2\tau_\varepsilon}(\prod_{j=1}^t|z_j|^2)^{-1}dV_{\CC^n} \\
        &\leq C''\prod_{j=1}^t\int_0^{r'}r_j^{2c_j+2\tau_\varepsilon-1}dr<+\infty.
    \end{align*}
    Therefore, $v$ is $L^2$-integrable.
Since $v$ is a $(p,0)$-form and $\Omega_{\varepsilon,\delta}\leq\Omega_{\varepsilon,1}$, we have $\|v\|^2_{\Omega_{\varepsilon,\delta},H_\varepsilon}\leq\|v\|^2_{\Omega_{\varepsilon,1},H_\varepsilon}$.
    
    Let $0<\lambda_{\varepsilon,\delta,1}\leq\cdots\leq\lambda_{\varepsilon,\delta,n}$ be eigenvalues of $\omega_\varepsilon$ with respect to $\Omega_{\varepsilon,\delta}$.
    We now show that 
    \begin{equation}\label{Formula:integral of eigenvalues}
        \begin{aligned}
            I_\varepsilon&:=\int_{Y_\varepsilon}(\lambda_{\varepsilon,0,1}+\cdots+\lambda_{\varepsilon,0,n-p})G_\varepsilon\Omega_\varepsilon^n \\
            &\leq 4(n-p)\varepsilon\int_{Y_\varepsilon}G_\varepsilon\Omega_\varepsilon^n   \\
            &\leq 4(n-p)\varepsilon\int_{\widetilde{X}_\varepsilon}\gamma_\varepsilon^n.
        \end{aligned}
    \end{equation}
    
    Fix $\delta>0$. Since $(Y_\varepsilon,\Omega_{\varepsilon,\delta})$ is complete, there are smooth compactly supported functions $\chi_\nu$ such that
    $$0\leq\chi_\nu\leq1,\quad
    \chi_\nu\rightarrow1,\quad \text{and} \quad
    |{d\chi_\nu}|_{\Omega_{\varepsilon,\delta}}\leq\nu^{-1}.$$
    Since $\chi_\nu v$ is a $(p,0)$-form, $\bar\partial^*(\chi_\nu v)=0$.
    We also have $\bar\partial(\chi_\nu v)=\bar\partial\chi_\nu\wedge v.$
    Note that 
    $$\sqrt{-1}\Theta_{\pi^*_\varepsilon L,H_\varepsilon^{-1}}\geq \omega_\varepsilon-4\varepsilon\Omega_\varepsilon\geq \omega_\varepsilon-4\varepsilon\Omega_{\varepsilon,\delta}.$$
    Therefore, acting on $(p,0)$-forms, we have
    $$\langle[\sqrt{-1}\Theta_{\pi^*_\varepsilon L^{-1},H_\varepsilon},\Lambda_{\Omega_{\varepsilon,\delta}}]v,v\rangle\geq(\lambda_{\varepsilon,\delta,1}+\cdots+\lambda_{\varepsilon,\delta,n-p}-4(n-p)\varepsilon)|v|^2_{\Omega_{\varepsilon,\delta},H_\varepsilon}.$$
    Then the Bochner--Kodaira--Nakano identity implies that
    \begin{equation}
        \begin{aligned}
            \int_{Y_\varepsilon}\chi_\nu^2(\lambda_{\varepsilon,\delta,1}+\cdots+\lambda_{\varepsilon,\delta,n-p}&-4(n-p)\varepsilon)|v|^2_{\Omega_{\varepsilon,\delta},H_\varepsilon}\Omega_{\varepsilon,\delta}^n \\
            &\leq\|\bar\partial\chi_\nu\wedge v\|^2_{\Omega_{\varepsilon,\delta},H_\varepsilon}\leq C\nu^{-2}\|v\|^2_{\Omega_{\varepsilon,\delta},H_\varepsilon},
        \end{aligned}
    \end{equation}
    where $C>0$ is a constant independent of $\nu$ and $v$.
    Hence 
    \begin{equation}
        \begin{aligned}
            \int_{Y_\varepsilon}\chi_\nu^2&(\lambda_{\varepsilon,\delta,1}+\cdots+\lambda_{\varepsilon,\delta,n-p})|v|^2_{\Omega_{\varepsilon,\delta},H_\varepsilon}\Omega_{\varepsilon,\delta}^n \\
            &\leq 4(n-p)\varepsilon \int_{Y_\varepsilon}\chi_\nu^2|v|^2_{\Omega_{\varepsilon,\delta},H_\varepsilon}\Omega_{\varepsilon,\delta}^n        +C\nu^{-2}\|v\|^2_{\Omega_{\varepsilon,\delta},H_\varepsilon},
        \end{aligned}
    \end{equation}
    Then Fatou's Lemma and dominated convergence theorem imply that, as $\nu\rightarrow\infty$,
    $$\int_{Y_\varepsilon}(\lambda_{\varepsilon,\delta,1}+\cdots+\lambda_{\varepsilon,\delta,n-p})|v|^2_{\Omega_{\varepsilon,\delta},H_\varepsilon}\Omega_{\varepsilon,\delta}^n
    \leq4(n-p)\varepsilon \int_{Y_\varepsilon}|v|^2_{\Omega_{\varepsilon,\delta},H_\varepsilon}\Omega_{\varepsilon,\delta}^n.$$
    Again, by Fatou's Lemma and dominated convergence theorem, as $\delta\rightarrow0$, we have 
    \begin{align*}
        \int_{Y_\varepsilon}(\lambda_{\varepsilon,0,1}+\cdots+\lambda_{\varepsilon,0,n-p})|v|^2_{\Omega_\varepsilon,H_\varepsilon}\Omega_\varepsilon^n
        &\leq4(n-p)\varepsilon \int_{Y_\varepsilon}|v|^2_{\Omega_\varepsilon,H_\varepsilon}\Omega_\varepsilon^n \\
        &=4(n-p)\varepsilon \int_{Y_\varepsilon}G_\varepsilon\Omega_\varepsilon^n \\
        &\leq 4(n-p)\varepsilon\int_{\widetilde{X}_\varepsilon}\gamma_\varepsilon^n.
    \end{align*}
    This proves formula~(\ref{Formula:integral of eigenvalues}).
    
    \medskip
    \noindent\textbf{Step 4: the determinant lower bound and the contradiction.}
    
    We now consider $\lambda_{\varepsilon,0,j}$.
    One has 
    $$\lambda_{\varepsilon,0,1}\cdots\lambda_{\varepsilon,0,n}=G_\varepsilon+R_\varepsilon.$$
    Denote 
    $$a_\varepsilon:=\frac{G_\varepsilon}{G_\varepsilon+R_\varepsilon}\leq 1 
    \quad \text{and} \quad
    S_\varepsilon:=\lambda_{\varepsilon,0,n-p+1}\cdots\lambda_{\varepsilon,0,n}.$$
    Then 
    $$G_\varepsilon+R_\varepsilon=\lambda_{\varepsilon,0,1}\cdots\lambda_{\varepsilon,0,n-p}S_\varepsilon\leq\lambda_{\varepsilon,0,n-p}^{n-p}S_\varepsilon.$$
    Hence, 
    $$\lambda_{\varepsilon,0,n-p}G_\varepsilon\geq a_\varepsilon(G_\varepsilon+R_\varepsilon)^{1+\frac{1}{n-p}}S_{\varepsilon}^{-\frac{1}{n-p}}.$$
    Write 
    $$a_\varepsilon(G_\varepsilon+R_\varepsilon)=\left(a_\varepsilon(G_\varepsilon+R_\varepsilon)^{1+\frac{1}{n-p}}S_\varepsilon^{-\frac{1}{n-p}}      \right)^{\frac{n-p}{n-p+1}}  
    \left(a_\varepsilon S_\varepsilon\right)^{\frac{1}{n-p+1}}.$$
    Then H\"older's inequality implies that
    \begin{equation}
        \begin{aligned}
            M_\varepsilon&=\int_{Y_\varepsilon}G_{\varepsilon}\Omega_\varepsilon^n             
            =\int_{Y_\varepsilon}a_\varepsilon(G_{\varepsilon}+R_\varepsilon)\Omega_\varepsilon^n  \\
            &\leq\left(\int_{Y_\varepsilon}a_\varepsilon(G_\varepsilon+R_\varepsilon)^{1+\frac{1}{n-p}}S_\varepsilon^{-\frac{1}{n-p}} \Omega_\varepsilon^n \right)^{\frac{n-p}{n-p+1}}
            \left(\int_{Y_\varepsilon}a_\varepsilon S_\varepsilon\Omega_\varepsilon^n \right)^{\frac{1}{n-p+1}}.
        \end{aligned}
    \end{equation}
    Therefore, 
    \begin{equation}
        \begin{aligned}
            I_\varepsilon&=\int_{Y_\varepsilon}(\lambda_{\varepsilon,0,1}+\cdots+\lambda_{\varepsilon,0,n-p})G_\varepsilon\Omega_\varepsilon^n \\
            &\geq \int_{Y_\varepsilon}\lambda_{\varepsilon,0,n-p}G_\varepsilon\Omega_\varepsilon^n \\
            &\geq\int_{Y_\varepsilon}a_\varepsilon(G_\varepsilon+R_\varepsilon)^{1+\frac{1}{n-p}}S_\varepsilon^{-\frac{1}{n-p}} \Omega_\varepsilon^n\\
            &\geq \frac{M_\varepsilon^{1+\frac{1}{n-p}}}{\left(\int_{Y_\varepsilon} S_\varepsilon\Omega_\varepsilon^n \right)^{\frac{1}{n-p}}}.
        \end{aligned}
    \end{equation}
    Since the product of the largest $p$ eigenvalues is bounded by the $p$-th elementary symmetric polynomial, there exists a constant $C_{n,p}$ such that
    \begin{align*}
        \int_{Y_\varepsilon} S_\varepsilon\Omega_\varepsilon^n&\leq 
        C_{n,p}\int_{Y_\varepsilon}\omega_\varepsilon^p\wedge\Omega_\varepsilon^{n-p} \\
        &=C_{n,p}\int_{Y_\varepsilon}(\gamma_\varepsilon+\ddbar\varphi_{\varepsilon})^p\wedge\Omega_\varepsilon^{n-p}          \\
        &=C_{n,p}\int_{\widetilde{X}_\varepsilon}\gamma_\varepsilon^p\wedge\Omega_\varepsilon^{n-p}\leq C_{n,p}C_0.
    \end{align*}
    Since $M_\varepsilon\geq\frac{1}{2}\int_{\widetilde{X}_\varepsilon}\gamma_\varepsilon^n$, we obtain
    $$I_\varepsilon\geq c_3\left(\int_{\widetilde{X}_\varepsilon}\gamma_\varepsilon^n\right)^{1+\frac{1}{n-p}}.$$
    On the other hand, we showed in Step 3 that
    $$I_\varepsilon\leq 4(n-p)\varepsilon\int_{\widetilde{X}_\varepsilon}\gamma_\varepsilon^n.$$
    Thus, 
    $$(\int_{\widetilde{X}_\varepsilon}\gamma_\varepsilon^n)^{\frac{1}{n-p}}\leq C_3\varepsilon$$
    for some constant $C_3>0$.
    Note that by the construction in Lemma~\ref{Lemma:Construct suitable metrics},
    $$\int_{\widetilde{X}_\varepsilon}\gamma_\varepsilon^n\geq c_0\varepsilon^{n-k}$$
    for any $k=\nd(L)$.
    Therefore,
    $$c_0\varepsilon^{n-k}\leq (C_3\varepsilon)^{n-p}.$$
    This is impossible when $\varepsilon$ tends to $0$ since $p<k$. 
    Thus this gives a contradiction. 
    Hence $H^0(X,\Omega^p_X(\log D)\otimes L^{-1})=0$.
    
    This completes the proof of Theorem~\ref{Main Theorem 1}.
    
\end{proof}

\begin{remark}
    We provide a simple proof via hyperplane induction on projective manifolds.
    Assume that $X$ is a projective manifold. 
    It is easy to see that the conclusion holds when $\dim X=1$.
    Assume that the conclusion holds for all projective manifolds with dimension at most $n-1$. 

    If $\nd(L)=n$, then $L$ is a big line bundle and $\kappa(L)=n$. 
    Then by Bogomolov--Sommese vanishing theorem (Theorem~\ref{Theorem:BS vanishing}), 
    $$H^0(X,\Omega^p_X(\log D)\otimes L^{-1})=0 \quad \text{for every } p<n.$$
    
    We now assume $\nd(L)<n$.
    Consider $s\in H^0(X,\Omega^p_X(\log D)\otimes L^{-1})$ with $p<\nd(L)$.
    Let $H$ be an ample divisor and take a general smooth $A\in|mH|$ for $m\gg0$ such that 
    $A+D$ is a simple normal crossing divisor and $\nd(L|_A)\geq\nd(L)$.
    Consider the exact sequence
    $$0\to \Omega_X^p(\log(D+A))\otimes\mathcal{O}_X(-A)\otimes L^{-1} \to \Omega_X^p(\log D)\otimes L^{-1} \to \Omega_A^p(\log D|_A)\otimes L^{-1}|_A\to 0.$$
    Then $s|_A\in H^0(A,\Omega^p_A(\log D|_A)\otimes L^{-1}|_A)=0$ by the induction hypothesis.
    Note that this is true for general $A$. 
    Hence $s=0$. 
    Therefore
    $$H^0(X,\Omega^p_X(\log D)\otimes L^{-1})=0 \quad \text{for every } p<\nd(L).$$
\end{remark}


    
\begin{proof}[Proof of Corollary~\ref{Corollary:Corollary with multiplier}]
    Note that $\mathcal{I}(\psi)$ is a torsion-free ideal subsheaf in $\mathcal{O}_X$.
    Then $\mathcal{I}(\psi)^{\vee\vee}$ is a rank-one reflexive sheaf. 
    Thus $M:=\mathcal{I}(\psi)^{\vee\vee}$ is a holomorphic line bundle.
    We have an inclusion $\mathcal{I}(\psi)^{\vee\vee}\hookrightarrow\mathcal{O}_X$ induced by the inclusion $\mathcal{I}(\psi)\hookrightarrow\mathcal{O}_X$.
    This inclusion gives a nonzero section $s\in H^0(X,M^{-1})$.
    Let $B=\sum_{i}a_iB_i$ be the zero divisor of section $s$. 
    We have $M^{-1}\simeq\mathcal{O}_X(B)$ and $M\simeq\mathcal{O}_X(-B)$.
    
    On the other hand, we have the following expression of $\mathcal{I}(\psi)$ by Proposition~\ref{Proposition:Demailly--Peternell} (\cite[Proposition~3.2]{DP03}). 
    Consider the Siu decomposition of the closed positive $(1,1)$-current $T:=\theta+\frac{\ddbar\psi}{2\pi}$:
    $$T=\sum_{j=1}^{\infty}\lambda_j[E_j]+R,$$
    where $E_j$ are effective divisors and $R$ is the residue closed positive current such that the Lelong sublevel sets $E_c(R)$, $c>0$, all have codimension two.
    Then we have 
    $$\mathcal{I}(\psi)\subset\mathcal{O}_X(-\sum\lfloor \lambda_j \rfloor E_j),$$
    and the equality holds on $X\setminus Z$ where $Z$ is an analytic subset of $X$ whose components all have codimension at least two.
    
    Therefore, we have $M=\mathcal{O}_X(-\sum\lfloor \lambda_j \rfloor E_j)$. 
    Consider the pseudo-effective line bundle $N:=L\otimes M$. 
    Since $$c_1(N)-\{\alpha\}=c_1(L)-\{\alpha\}+c_1(M)=\{T-\sum\lfloor \lambda_j \rfloor E_j\}$$
    is pseudo-effective, 
    we have $\nd(c_1(N))\geq\nd(\alpha)$.
    By Serre duality, 
    \begin{align*}
        &H^n(X,\Omega^p_X(\log D)\otimes L\otimes\mathcal{O}_X(-D)\otimes\mathcal{I}(\psi))^*\\
        \cong &\Hom(\Omega^p_X(\log D)\otimes L\otimes\mathcal{O}_X(-D)\otimes\mathcal{I}(\psi),K_X)\\
        \cong& H^0(X,\Omega_X^{n-p}(\log D)\otimes L^{-1}\otimes\mathcal{I}(\psi)^\vee) \\
        \cong& H^0(X,\Omega^{n-p}_X(\log D)\otimes N^{-1}) 
        =0
    \end{align*}
    for $n-p<\nd(N)$, equivalently, $p\geq n-\nd(N)+1$.
    Therefore, 
    $$H^n(X,\Omega^p_X(\log D)\otimes L\otimes\mathcal{O}_X(-D)\otimes\mathcal{I}(\psi))=0 $$
    for $p\geq n-\nd(\{\alpha\})+1$.
\end{proof}

\section{Applications} \label{Section:Applications}

In this section, we give some geometric applications of Theorem~\ref{Main Theorem 1}. 
We first show the following useful reformulation.

\begin{proposition}\label{Proposition:Numerical line subbundle principle}
    Let $X$ be a compact K\"ahler manifold and $D$ a reduced simple normal crossing divisor on $X$.
    Let $L$ be a pseudo-effective line bundle on $X$.
    If there exists a nonzero morphism
    $$L\longrightarrow \Omega_X^p(\log D),$$
    then
    $$\nd(L)\leq p.$$
\end{proposition}
\begin{proof}
    A nonzero morphism $L\longrightarrow\Omega_X^p(\log D)$ is equivalent to a nonzero section
    $$0\neq s\in H^0(X,\Omega_X^p(\log D)\otimes L^{-1}).$$
    If $p<\nd(L)$, this contradicts Theorem~\ref{Main Theorem 1}.
\end{proof}

Then we have the following direct conclusion.

\begin{corollary}\label{Corollary:log-conormal-multivectors}
    Let $X$ be a compact K\"ahler manifold of dimension $n$ and $D$ a reduced simple normal crossing divisor on $X$.
    The following statements hold.
    \begin{enumerate}
        \item Let $\mathcal{F}\subset T_X(-\log D)$ be a saturated coherent subsheaf of corank $r$, where $1\leq r\leq n$, and set
        $$\mathcal{Q}:=T_X(-\log D)/\mathcal{F}, \quad
        A_{\mathcal{F},D}:=\det(\mathcal{Q}^*)=(\wedge^r\mathcal{Q}^*)^{**}.$$
        If $A_{\mathcal{F},D}$ is pseudo-effective, then
        $\nd(A_{\mathcal{F},D})\leq r$.
        In particular, in the case of $r=1$, suppose that $\mathcal{F}\subset T_X$ is a saturated codimension-one foliation, and that every irreducible component of $D$ is $\mathcal{F}$-invariant. 
        Denote $$N_{\mathcal{F}}:=(T_X/\mathcal{F})^{**}.$$ 
        Then $A_{\mathcal{F},D}\simeq N_{\mathcal{F}}^*\otimes\mathcal{O}_X(D)$.
        If $N_{\mathcal{F}}^*\otimes\mathcal{O}_X(D)$ is pseudo-effective, then
        $$\nd(N_{\mathcal{F}}^*\otimes\mathcal{O}_X(D))\leq 1.$$
        
        \item Assume that $K_X+D$ is pseudo-effective. Then
        $$H^0(X,\wedge^pT_X(-\log D))=0 \quad \text{for every } n-p<\nd(K_X+D).$$
        In particular, if $\nd(K_X+D)=n$, then
        $$H^0(X,T_X(-\log D))=0.$$
        Thus $(X,D)$ has no nonzero infinitesimal automorphism.
        If $\nd(K_X+D)\geq n-1$, then
        $$H^0(X,\wedge^2T_X(-\log D))=0.$$
        Hence there is no nonzero logarithmic Poisson structure on $(X,D)$.
    \end{enumerate}
\end{corollary}

\begin{proof}
    (1) The natural inclusion
    $$\mathcal{Q}^*\hookrightarrow \Omega_X^1(\log D)$$
    induces, outside an analytic subset of codimension at least two, a nonzero morphism
    $$A_{\mathcal{F},D}\longrightarrow \wedge^r\Omega_X^1(\log D)=\Omega_X^r(\log D).$$
    Since $A_{\mathcal{F},D}$ and $\Omega_X^r(\log D)$ are locally free, this morphism extends across the analytic subset of codimension at least two.
    Proposition~\ref{Proposition:Numerical line subbundle principle} applied with $p=r$ gives the desired conclusion.

    In the codimension-one case, the invariance of every
    irreducible component of $D$ gives
    $\mathcal{F}\subset T_X(-\log D)$. Using
    \[
        \det T_X(-\log D)
       \simeq
       \det T_X\otimes\mathcal{O}_X(-D),
    \]
    we obtain
    \[
        \det\left(T_X(-\log D)/\mathcal{F}\right)
        \simeq
        N_{\mathcal{F}}\otimes\mathcal{O}_X(-D),
    \]
    and hence
    \[
        A_{\mathcal{F},D}
        \simeq
        N_{\mathcal{F}}^*\otimes\mathcal{O}_X(D).
    \]
    
    (2) Logarithmic duality gives
    $$\wedge^pT_X(-\log D)\simeq\Omega_X^{n-p}(\log D)\otimes\mathcal{O}_X(-(K_X+D)).$$
    Thus a nonzero logarithmic $p$-vector field determines a nonzero section of
    $$\Omega_X^{n-p}(\log D)\otimes\mathcal{O}_X(-(K_X+D)).$$
    By applying Theorem~\ref{Main Theorem 1} to $K_X+D$, we obtain the desired conclusion.
\end{proof}

    No Frobenius integrability assumption is used in the first
    assertion. For projective $X$, its codimension-one consequence
    already appears in \cite[Proposition~9.3]{Tou16}; the formulation
    above records the compact K\"ahler numerical bound uniformly in
    arbitrary corank.

    Here a logarithmic Poisson structure means a section
    \[
        \sigma\in
        H^0\left(X,\wedge^2T_X(-\log D)\right)
    \]
    satisfying $[\sigma,\sigma]=0$. The corollary does not exclude
    logarithmic symplectic structures. Indeed, if $n=2m$ and $D$ is
    the reduced degeneracy divisor of a logarithmic symplectic
    structure $\sigma$, then
    \[
        \operatorname{div}(\sigma^m)=D,
        \qquad
        \mathcal{O}_X(D)\simeq K_X^{-1},
    \]
    so that $K_X+D\sim 0$ and the numerical threshold above gives no
    vanishing; see \cite{Pym18}.

Combining Proposition~\ref{Proposition:Numerical line subbundle principle} and the results in \cite{CP16, IM22}, we have the following result.

\begin{theorem}\label{Theorem:Logarithmic c2 rigidity}
    Let $X$ be a compact K\"ahler manifold and $D$ a reduced simple normal crossing divisor on $X$.
    Assume that $K_X+D$ is nef.
    Then the following are equivalent:
    \begin{enumerate}
        \item $c_2(\Omega^1_X(\log D))=0 \text{ in } H^{2,2}(X,\mathbb{R})$.
        \item $\Omega^1_X(\log D)$ is nef and $\nd(K_X+D)\leq1$.
    \end{enumerate}

\end{theorem}
\begin{proof}
    If $K_X+D$ is nef and $c_2(\Omega^1_X(\log D))=0$, then by \cite[Theorem~1.1 and formula (5.3)]{CP16}, for any K\"ahler form $\eta$ and torsion-free quotient $\Omega^1_X(\log D)\twoheadrightarrow Q$, we have 
    $$c_1(Q)\cdot\{\eta\}^{n-1}\geq 0.$$
    This means that $\Omega^1_X(\log D)$ is $\{\eta\}^{n-1}$-generically nef (see \cite{IM22}). 
    Then by \cite[Theorem~1.6 and Theorem~5.7]{IM22}, $\Omega^1_X(\log D)$ is nef.
    We prove $\nd(K_X+D)\leq1$ by contradiction.
    Assume that
    $$\nd(K_X+D)\geq2.$$
    By \cite[Theorem~5.7]{IM22}, there exists a line bundle $L\subset \Omega^1_X(\log D)$ such that $\Omega^1_X(\log D)/L$ is torsion-free and
    $$c_1(L)=c_1(\Omega^1_X(\log D))=c_1(K_X+D).$$
    In particular, $L$ is nef and
    $$\nd(L)=\nd(K_X+D)\geq2.$$
    On the other hand, the inclusion
    $$L\hookrightarrow \Omega_X^1(\log D)$$
    and Proposition~\ref{Proposition:Numerical line subbundle principle} imply that $\nd(L)\leq1$, which is a contradiction.

    If $\Omega^1_X(\log D)$ is nef and $\nd(K_X+D)\leq1$, then from \cite[Corollary~2.6]{DPS94}, we obtain 
    $$0\leq c_2(\Omega^1_X(\log D))\cdot\eta^{n-2}\leq c_1(\Omega^1_X(\log D))^2\cdot\eta^{n-2}=0$$
    for any K\"ahler form $\eta$. 
    Again by the first part of the proof, $c_2(\Omega^1_X(\log D))=0$.

\end{proof}

\begin{example}\label{Example:Sharpness of logarithmic c2 rigidity}
    Let $C$ be a compact Riemann surface of genus $\geqslant 1$ and let $D_C$ be a reduced effective divisor such that
    $$\deg(K_C+D_C)>0.$$
    Let $T$ be a complex torus of dimension $n-1$ and set
    $$X=C\times T, \qquad D=D_C\times T.$$
    Then
    $$\Omega_X^1(\log D)\simeq \operatorname{pr}_C^*\Omega_C^1(\log D_C)\oplus\mathcal{O}_X^{\oplus(n-1)}$$
    is nef and
    $$c_2(\Omega_X^1(\log D))=0.$$
    Moreover,
    $$K_X+D\simeq \operatorname{pr}_C^*(K_C+D_C)$$
    and
    $$\nd(K_X+D)=1.$$
    Hence the upper bound in Theorem~\ref{Theorem:Logarithmic c2 rigidity} is sharp.
\end{example}


\begin{thebibliography}{99}

    \bibitem[Bog78]{Bog78} Fedor Alekseevich Bogomolov, \textit{Holomorphic tensors and vector bundles on projective varieties}, Izv. Akad. Nauk SSSR Ser. Mat. {\bf 42} (1978), no.~6, 1227--1287, 1439.
    
    \bibitem[Bog80]{Bog80} Fedor Alekseevich Bogomolov, \textit{Unstable vector bundles and curves on surfaces}, in Proceedings of the International Congress of Mathematicians (Helsinki, 1978), pp. 517--524, Acad. Sci. Fennica, Helsinki, 1980.
    
    \bibitem[Bou02]{Bou02} S\'ebastien Boucksom, \textit{C\^{o}nes positifs des vari\'et\'es complexes compactes}, PhD thesis, Universit\'e Joseph Fourier, Grenoble, 2002.
    
    \bibitem[BEGZ10]{BEGZ10} S\'ebastien Boucksom, Philippe Eyssidieux, Vincent Guedj, and Ahmed Zeriahi, \textit{Monge-Amp\`ere equations in big cohomology classes}, Acta Math. {\bf 205} (2010), no.~2, 199--262.

    \bibitem[CP16]{CP16} Fr\'ed\'eric Campana and Mihai P\u{a}un, \textit{Positivity properties of the bundle of logarithmic tensors on compact K\"ahler manifolds}, Compos. Math. {\bf 152} (2016), no.~11, 2350--2370.

    \bibitem[Cao14]{Cao14} Junyan Cao, \textit{Numerical dimension and a Kawamata-Viehweg-Nadel-type vanishing theorem on compact K\"ahler manifolds}, Compos. Math. {\bf 150} (2014), no.~11, 1869--1902.

    \bibitem[Dem]{DemBig} Jean-Pierre Demailly, \textit{Complex analytic and differential geometry}, 2012, available at \url{http://www-fourier.ujf-grenoble.fr/~demailly/books.html}

    \bibitem[Dem02]{Dem02} Jean-Pierre Demailly, \textit{On the Frobenius integrability of certain holomorphic $p$-forms}, in \textit{Complex geometry (G\"ottingen, 2000)}, Springer, Berlin, 2002, 93--98.

    \bibitem[Dem12]{DemSmall} Jean-Pierre Demailly, \textit{Analytic methods in algebraic geometry}, Surveys of Modern Mathematics {\bf 1}, International Press, Somerville, MA; Higher Education Press, Beijing, 2012.

    \bibitem[DP10]{DP10} Jean-Pierre Demailly and Nefton Pali, \textit{Degenerate complex Monge-Amp\`ere equations over compact K\"ahler manifolds}, Internat. J. Math. {\bf 21} (2010), no.~3, 357--405.
    
    \bibitem[DP03]{DP03} Jean-Pierre Demailly and Thomas Peternell, \textit{A Kawamata-Viehweg vanishing theorem on compact K\"ahler manifolds}, J. Differential Geom. {\bf 63} (2003), no. 2, 231--277.

    \bibitem[DPS94]{DPS94} Jean-Pierre Demailly, Thomas Peternell, and Michael Schneider, \textit{Compact complex manifolds with numerically effective tangent bundles}, J. Algebraic Geom. {\bf 3} (1994), no.~2, 295--345.

    \bibitem[EV92]{EV92} H\'el\`ene Esnault and Eckart Viehweg, \textit{Lectures on vanishing theorems}, DMV Seminar, 20, Birkh\"auser, Basel, 1992. 
    
    \bibitem[Gra15]{Gra15} Patrick Graf, \textit{Bogomolov-Sommese vanishing on log canonical pairs}, J. Reine Angew. Math. {\bf 702} (2015), 109--142.
    
    \bibitem[GZ15]{GZ15} Qi'an Guan and Xiangyu Zhou, \textit{A proof of Demailly's strong openness conjecture}, Ann. of Math. (2) {\bf 182} (2015), no. 2, 605--616.

    \bibitem[IM22]{IM22} Masataka Iwai and Shin-ichi Matsumura, \textit{Abundance theorem for minimal compact K\"ahler manifolds with vanishing second Chern class}, Preprint, arXiv:2205.10613 (2022).

    \bibitem[LMNWZ25]{LMNWZ25} Zhi Li, Xiankui Meng, Jiafu Ning, Zhiwei Wang, and Xiangyu Zhou, \textit{On a Bogomolov type vanishing theorem}, Nagoya Math. J. {\bf 257} (2025), 170--182.

    \bibitem[MQZ26]{MQZ26} Xiankui Meng, Chenghao Qing, and Xiangyu Zhou, \textit{A Bogomolov type vanishing theorem}, Preprint, arXiv:2607.14858 (2026).
    
    \bibitem[Mou98]{Mou98} Christophe Mourougane, \textit{Versions k\"ahl\'eriennes du th\'eor\`eme d'annulation de Bogomolov}, Collect. Math. {\bf 49} (1998), no.~2-3, 433--445.
    
    \bibitem[Nad90]{Nad90} Alan Michael Nadel, \textit{Multiplier ideal sheaves and K\"ahler-Einstein metrics of positive scalar curvature}, Ann. of Math. (2) {\bf 132} (1990), no. 3, 549--596.

    \bibitem[PRT22]{PRT22} Jorge Vit\'orio Pereira, Erwan Rousseau, and Fr\'ed\'eric Touzet, \textit{Numerically non-special varieties}, Compos. Math. {\bf 158} (2022), no.~6, 1428--1447.

    \bibitem[Pym18]{Pym18} Brent Pym, \textit{Constructions and classifications of projective Poisson varieties}, Lett. Math. Phys. {\bf 108} (2018), no.~3, 573--632.
    
    \bibitem[SS85]{SS85} Bernard Shiffman and Andrew John Sommese, \textit{Vanishing theorems on complex manifolds}, Progress in Mathematics, 56, Birkh\"auser Boston, Boston, MA, 1985.

    \bibitem[Tou16]{Tou16} Fr\'ed\'eric Touzet, \textit{On the structure of codimension one foliations with pseudoeffective conormal bundle}, in \textit{Foliation theory in algebraic geometry}, Simons Symposia, Springer, Cham, 2016, 157--216.
    
    \bibitem[Wat23]{Wat23} Yuta Watanabe, \textit{Bogomolov-Sommese type vanishing theorem for holomorphic vector bundles equipped with positive singular Hermitian metrics}, Math. Z. {\bf 303} (2023), no. 4, Paper No. 92, 23 pp.
    
    \bibitem[Wat26]{Wat26} Yuta Watanabe, \textit{$L^2$-Dolbeault isomorphisms and vanishing theorems for logarithmic sheaves twisted by multiplier ideal sheaves}, Math. Z. {\bf 312} (2026), no.~2, Paper No. 40, 28 pp.
    
    \bibitem[Wu20]{Wu20} Xiaojun Wu, \textit{On a vanishing theorem due to Bogomolov}, Preprint, arXiv:2011.13751 (2020).

    \bibitem[Yau78]{Yau78} Shing-Tung Yau, \textit{On the Ricci curvature of a compact K\"ahler manifold and the complex Monge-Amp\`ere equation. I}, Comm. Pure Appl. Math. {\bf 31} (1978), no.~3, 339--411.
    
    \bibitem[Zuc79]{Zuc79} Steven Zucker, \textit{Hodge theory with degenerating coefficients. $L^2$ cohomology in the Poincar\'{e} metric}, Ann. of Math. (2) {\bf 109} (1979), no. 3, 415--476.

\end{thebibliography}
\end{document}